\documentclass[11pt,twoside]{article} 
\usepackage[hang,flushmargin]{footmisc}
\usepackage{bbm}
\usepackage{amsmath,amssymb}
\makeatletter
\newcommand*{\mint}[1]{%
  \mint@l{#1}{}%
}
\newcommand*{\mint@l}[2]{%
  \@ifnextchar\limits{%
    \mint@l{#1}%
  }{%
    \@ifnextchar\nolimits{%
      \mint@l{#1}%
    }{%
      \@ifnextchar\displaylimits{%
        \mint@l{#1}%
      }{%
        \mint@s{#2}{#1}%
      }%
    }%
  }%
}
\newcommand*{\mint@s}[2]{%
  \@ifnextchar_{%
    \mint@sub{#1}{#2}%
  }{%
    \@ifnextchar^{%
      \mint@sup{#1}{#2}%
    }{%
      \mint@{#1}{#2}{}{}%
    }%
  }%
}
\def\mint@sub#1#2_#3{%
  \@ifnextchar^{%
    \mint@sub@sup{#1}{#2}{#3}%
  }{%
    \mint@{#1}{#2}{#3}{}%
  }%
}
\def\mint@sup#1#2^#3{%
  \@ifnextchar_{%
    \mint@sup@sub{#1}{#2}{#3}%
  }{%
    \mint@{#1}{#2}{}{#3}%
  }%
}
\def\mint@sub@sup#1#2#3^#4{%
  \mint@{#1}{#2}{#3}{#4}%
}
\def\mint@sup@sub#1#2#3_#4{%
  \mint@{#1}{#2}{#4}{#3}%
}
\newcommand*{\mint@}[4]{%
  \mathop{}%
  \mkern-\thinmuskip
  \mathchoice{%
    \mint@@{#1}{#2}{#3}{#4}%
        \displaystyle\textstyle\scriptstyle
  }{%
    \mint@@{#1}{#2}{#3}{#4}%
        \textstyle\scriptstyle\scriptstyle
  }{%
    \mint@@{#1}{#2}{#3}{#4}%
        \scriptstyle\scriptscriptstyle\scriptscriptstyle
  }{%
    \mint@@{#1}{#2}{#3}{#4}%
        \scriptscriptstyle\scriptscriptstyle\scriptscriptstyle
  }%
  \mkern-\thinmuskip
  \int#1%
  \ifx\\#3\\\else_{#3}\fi
  \ifx\\#4\\\else^{#4}\fi
}
\newcommand*{\mint@@}[7]{%
  \begingroup
    \sbox0{$#5\int\m@th$}%
    \sbox2{$#5\int_{}\m@th$}%
    \dimen2=\wd0 %
    \let\mint@limits=#1\relax
    \ifx\mint@limits\relax
      \sbox4{$#5\int_{\kern1sp}^{\kern1sp}\m@th$}%
      \ifdim\wd4>\wd2 %
        \let\mint@limits=\nolimits
      \else
        \let\mint@limits=\limits
      \fi
    \fi
    \ifx\mint@limits\displaylimits
      \ifx#5\displaystyle
        \let\mint@limits=\limits
      \fi
    \fi
    \ifx\mint@limits\limits
      \sbox0{$#7#3\m@th$}%
      \sbox2{$#7#4\m@th$}%
      \ifdim\wd0>\dimen2 %
        \dimen2=\wd0 %
      \fi
      \ifdim\wd2>\dimen2 %
        \dimen2=\wd2 %
      \fi
    \fi
    \rlap{%
      $#5%
        \vcenter{%
          \hbox to\dimen2{%
            \hss
            $#6{#2}\m@th$%
            \hss
          }%
        }%
      $%
    }%
  \endgroup
}

\usepackage{mathrsfs}
\usepackage{amssymb}
\usepackage{amsmath}
\usepackage{amsthm}
\usepackage{amsfonts}
\usepackage{color}
\usepackage{graphicx}
\usepackage[active]{srcltx}
\usepackage{tikz}
\usepackage{pgflibraryarrows}
\usepackage{pgflibrarysnakes}

\usepackage[cp1252]{inputenc}

\usepackage{mathrsfs}
\usepackage{graphicx}

\allowdisplaybreaks

\usepackage{titletoc}
\titlecontents{section}[0pt]{\addvspace{2pt}\filright}
              {\contentspush{\thecontentslabel\ }}
              {}{\titlerule*[8pt]{.}\contentspage}

\def\rr{{\mathbb R}}
\def\rn{{{\rr}^n}}

\def\nn{{\mathbb N}}

\def\fz{\infty}

\def\dist{{\mathop\mathrm{\,dist\,}}}
\def\loc{{\mathop\mathrm{\,loc\,}}}

\def\dz{\delta}
\def\bdz{\Delta}

\def\bz{\beta}

\def\gz{{\gamma}}

\def\wz{\widetilde}

\def\bint{{\ifinner\rlap{\bf\kern.35em--}
\int\else\rlap{\bf\kern.45em--}\int\fi}\ignorespaces}

\def\bbint{{\ifinner\rlap{\bf\kern.35em--}
\hspace{0.078cm}\int\else\rlap{\bf\kern.45em--}\int\fi}\ignorespaces}

\def\r{\right}
\def\lf{\left}

\newtheorem{thm}{Theorem}[section]
\newtheorem{lem}[thm]{Lemma}
\newtheorem{prop}[thm]{Proposition}
\newtheorem{rem}[thm]{Remark}
\newtheorem{cor}[thm]{Corollary}
\numberwithin{equation}{section}

\title
{\Large\bf  Optimal regularity \& Liouville property
for stable solutions to semilinear elliptic equations in $\mathbb R^n$ with  $n\ge10$

\footnotetext{\hspace{-0.35cm}\noindent{2020 {\it Mathematics Subject Classification:}\ 35J61}\endgraf
{\it Key words and phases:}  semilienar elliptic equation, stable solution, BMO regularity, Morry regularity\endgraf
The second author is partially funded by the Chinese Academy of Science and NSFC grant No. 11688101.\endgraf The first and third authors are supported by the National Natural Science Foundation of China (No. 11871088 \& No.12025102).\endgraf $^\ast$ Corresponding author.}
}

\author{Fa Peng, Yi Ru-Ya Zhang$^\ast$ and Yuan Zhou
}

\begin{document}

\arraycolsep=1pt
\allowdisplaybreaks
 \maketitle

\begin{center}
\begin{minipage}{13.5cm}\small
 \noindent{\bf Abstract.}\quad Let $ 0\le f\in C^{0,1}(\rr)$. 
Given a  domain $\Omega\subset \mathbb R^n$,  we prove that any stable solution to the equation $-\Delta u=f(u)$ in $\Omega$ satisfies
\begin{itemize}
\item a BMO interior regularity when  $n=10$,

\item an Morrey $M^{p_n,4+2/(p_n-2)}$ interior  regularity when $n\ge11$, where $$p_n=\frac{2(n-2\sqrt{n-1}-2)}{n-2\sqrt{n-1}-4}. $$
 \end{itemize}
This result is optimal as hinted by e.g. \cite{bv, CC2006, d11}, and answers an open question raised by Cabr\'e, Figalli, Ros-Oton and Serra \cite{cf}. As an application, we show a sharp Liouville property: Any stable solution $u \in C^2(\rn)$ to  $-\Delta u=f(u)$ in $\rn$ satisfying the growth condition
 \begin{equation*}|u(x)|=\lf\{\begin{array}{ll}
o\left( \log|x|  \right) \quad&{\rm as}\ |x|\to+\fz,\quad   \quad{\rm when}\quad n=10;\\
 o\left( |x| ^{ -\frac n2+\sqrt{n-1}+2  }\right) \quad &{\rm as}\ |x|\to+\fz, \quad  \quad{\rm when}\quad n\ge11 
\end{array}\r.
\end{equation*} 
must be a constant. This extends the well-known Liouville property for radial stable solutions  obtained by Villegas \cite{v07}.
\end{minipage}
\end{center}

\section{Introduction}

\par\indent

Let  $\Omega$ be a bounded domain of $\rr^n$ with $n\ge2$.
Given any  local Lipscitz function $f\colon \rr\to \mathbb R$  (for short $f\in C^{0,1}(\rr)$), we consider the semilinear elliptic equation
\begin{equation}\label{st-eq}
-\bdz u=f(u)\quad{\rm in}\quad \Omega,
\end{equation}
which  is the Euler-Lagrange equation for the energy functional
\begin{equation}\label{en}
{\mathcal E}(u):=\int_{\Omega}\left(\frac12|Du|^2-F(u)\right)\,dx,
\end{equation}
where $F(t)=\int_0^tf(s)\,ds$ for $t\in \rr$.
A function  $u\in W^{1,2}(\Omega)$ is called as a weak solution to the equation \eqref{st-eq} if
  $f(u)\in L^1_\loc(\Omega)$ and
$$\int_\Omega Du\cdot D\xi\,dx=\int_\Omega f(u)\xi\,dx\quad\forall \xi\in C_c^\fz(\Omega),$$
that is,  $u$ is a critical point of the energy functional ${\mathcal E}$.
We say that a weak solution $u$ 
 is \emph{stable} in $\Omega$ if
$f'_-(u)\in L^1_\loc(\Omega)$ and
\begin{equation}\label{st-in}
\int_{\Omega}f'_-(u)\xi^2\,dx\le \int_{\Omega}|D\xi|^2\,dx\quad \forall \xi\in
C^\fz_c(\Omega),
\end{equation}
that is, the second variation of the energy functional $\mathcal E $ is nonnegative.
Here and below,
$$f'_-(t)=\liminf_{h\to 0}\frac{f(t+h)-f(t)}{h}\quad\forall t\in\rr,$$
and note that $f'_-(t)=f'(t)$  whenever $f \in C^1(\rr)$.


The study of stable solutions to semilinear elliptic equations can be traced back to the seminal paper  \cite{cr75} by   Crandall and Rabinowitz
in 1975.   The regularity of   stable solutions    provides
an important way to understand the regularity of extremal solution $u^{\star}$ to the Gelfand-type problem
\begin{align}\label{ext}
\left\{
\begin{aligned}
&-\bdz u=\lambda^{\star} f(u)&{\rm in}\ &\Omega\\
&\quad u>0&{\rm in}\ &\Omega\\
&\quad u=0&{\rm on}\ &\partial \Omega
\end{aligned}
\right.
\end{align}
for some positive constant $\lambda^{\star}>0$. We refer to  \cite{b03,c17,g63} and the reference therein for a comprehensive analysis of \eqref{ext} and related topics.
Note that the extremal solution $u^\star$ can be
can be approximated by stable solutions $\{u_\lambda\}_{\lambda<\lambda^{\star}}$, see
e.g.\ \cite{d11}.

In dimension $n\le 9$,
 Brezis \cite{b03} asked an open problem: whether the extremal solution $u^\star$ to equation \eqref{ext} is bounded for some $f$ and $\Omega$.
Since $u^\star$ is approximated by stable soluitons $\{u_\lambda\}_{\lambda<\lambda^{\star}}$, it suffices to   estbalish  some apriori bound for  stable solutions.
 In  recent years, there were several strong efforts to study regularity  for
stable solutions  and  hence for Brezis' open problem.
In particular, a positive answer   was given
by Nedev \cite{n20} when $n\le 3$ and by Cabr\'e \cite{c10} when $n=4$ (see also \cite{c19} for an alternative proof).

Very recently, Cabr\'e, Figalli, Ros-Oton and
Serra \cite{cf} provide a complete answer to Brezis'  open problem  when $f\ge0$ based on certain Morrey-type estimate for $n\ge 3$.
Throughout this paper, for $p\in[1,\fz)$ and $\beta\in(0,n)$, we  define the Morrey  norm  
\begin{equation}\label{b-m}
\quad \|w\| _{M^{p,\bz}(\Omega)}:=\sup_{y\in \Omega,r>0}\left(r^{\bz-n}\int_{\Omega\cap B_r(y)}|w|^p\,dx\right)^{1/p}<\fz,
\end{equation}
where $B_r(y) $ denotes the ball with center $y$ and radius $r>0$.
We simply write $B_r$ when the center of the ball is at the origin.
In addition, following the convention, we denote by $C(a,b,\cdots)$   a positive constant depending  only on the parameters $a$, $b$, $\cdots$.

In dimension $n\ge10$, in particular, Cabr\'e, Figalli, Ros-Oton and Serra \cite[Theorem 1.9]{cf}  established the following regularity
of stable solutions to the equation \eqref{st-eq}.

\begin{thm} [\cite{cf}]\label{cf}
Suppose that  $ f\in C^{0,1}(\rr)$ is  nonnegative.   If    $u\in C^2(B_1)$ is a stable solution to \eqref{st-eq} in $B_1$,
then
\begin{align}\label{es-cf}
\|u\|_{M^{p,2+\frac{4}{p-2}}{(B_{1/2})}}\le C(n,p)\|u\|_{L^1(B_1)}\quad
{\rm for\ every}\quad p<p_n.
\end{align}
where
\begin{align}\label{inde}
p_n:=\left\{
\begin{aligned}
&\fz\quad&{\rm if}\ n=10,\\
&\frac{2(n-2\sqrt{n-1}-2)}{n-2\sqrt{n-1}-4}\quad&{\rm if}\ n\ge11.
\end{aligned}
\right.
\end{align}

Moreover,  suppose additionally that $f $ is  nondecreasing,  and
$\Omega$ be a bound domain of class $C^3$.
If   $u\in C^2(\Omega)\cap C^0(\overline \Omega)$ is a stable solution
to \eqref{st-eq} in $\Omega$ with boundary $u=0$ on $\partial \Omega$, then
\begin{equation}\label{es-cf1}
\|u\|_{M^{p,2+\frac{4}{p-2}}{(\Omega)}}\le C(n,p,\Omega)\|u\|_{L^1(\Omega)}\quad
{\rm for\ every}\quad p<p_n.
\end{equation}
\end{thm}

We remark that the exponent 
$n-2\sqrt{n-1}-4$ changes sign when $n=10$, which was already appeared in e.g.\ \cite{gnw}. 

However for the endpoint case $p=p_n$, Cabr\'e, Figalli, Ros-Oton and Serra \cite[Section 1.3]{cf} pointed out  that
it is an open question whether \eqref{es-cf} holds.

%

As hinted by earlier results \cite{CC2006} in the radial symmetric case, when $n=10$,   instead of $L^\infty=M^{\fz,2}$  a more suitable space to consider is class of functions with  bounded mean oscillations (BMO space) as  remarked therein.
Indeed,
  $u(x)=-2\log|x|$ is a stable solution to the equation \eqref{st-eq} in $B_1$ with  $f(u)=2(n-2)e^u$.
Obviously,
$u\in {\rm BMO}(B_1)$ but $u\notin L^\fz(B_1)$.  Here and below,
the {\rm BMO} norm is defined as
$$\|u\|_{{\rm BMO}(\Omega)}:=\sup_{y\in \Omega,r>0} \inf_{c\in\rr}\mint{-}_{\Omega\cap B_r(y)}
\left|u(x)-c\right|\,dx,$$
where, $\mint-_Ev\,dx$ denotes the integral average of $v$ on a measurable set $E$.

On the other hand, when $n\ge11$, also hinted by the results in \cite{CC2006}, in  the range $p\le p_n$ is the best possible to get \eqref{es-cf}.   Besides, it was proven in  \cite{bv}   that the function
$u(x)=|x|^{-\frac{2}{q_n-1}}-1$ is the extremal solution to
\begin{align}\label{ex}
\begin{aligned}
&-\bdz u=\lambda^{\star}(1+u)^{q_n}\quad{\rm in}\ B_1;\quad u=0\quad\quad {\rm on}\ \partial B_1
\end{aligned}
\end{align}
with $\lambda^{\star}=\frac{2}{q_n}$ and $q_n:=\frac{n-2\sqrt{n-1}}{n-2\sqrt{n-1}-4}$.
  It is easy to see that
$u\in M^{p,2+\frac{4}{p-2}}(B_{1/2})$ if and only if $p\le p_n$. Recall  that by \cite[Section 3.2.2]{d11},  such extremal solution    can be approximated by
stable solutions.


The first main purpose of this paper is to establish the following regularity at the end-point $p_n$
for stable solutions to the equation \eqref{st-eq} when $n\ge10$, and then  answers the above open question
 by Cabr\'e, Figalli, Ros-Oton and Serra \cite{cf}.
\begin{thm}\label{th}
Suppose $f\in C^{0,1}(\rr)$ is  nonnegative.
For any  stable solution $u\in C^2(B_1)$
to \eqref{st-eq} in $B_1$,  when $n=10$ we have
 \begin{align}\label{tho-1}
\|u\|_{{\rm BMO}(B_{1/2})}\le C(n)\|u\|_{L^1(B_1)},
\end{align}
and when $n\ge 11$ we have
\begin{align}\label{tho-2}
 \|u\|_{M^{p_n,2+\frac{4}{p_n-2}}{(B_{1/2})}}\le C(n) \|u\|_{L^1(B_1)}.
 \end{align}

 Moreover,  suppose additionally that $f $ is  nondecreasing,  and
$\Omega$ is a bounded smooth   convex domain.
For any positive  stable solution
 $u\in C^2(\overline \Omega)$  to \eqref{st-eq} with boundary
 $u=0$ on $\partial \Omega$,   when $n=10$ we have
 \begin{align}\label{tho-3}
\|u\|_{{\rm BMO}(\Omega)}\le C(n,\Omega)\|u\|_{L^1(\Omega)},
\end{align}
and when $n\ge 11$ we have
 \begin{align}\label{tho-4}
 \|u\|_{M^{p_n,2+\frac{4}{p_n-2}}{(\Omega)}}\le C(n,\Omega)\|u\|_{L^1(\Omega)}.
 \end{align}
 \end{thm}

As a direct consequence of the above a priori estimates, we have the following result for stable solution in $W^{1,2}$.
\begin{cor}\label{cor}
Suppose that   $\Omega\subset \mathbb R^n$ is a bounded smooth convex domain and
 that $f\in C^{0,1}(\rr)$ is nonnegative, nondecreasing, convex, and satisfies
$ f(t)/t \to+\fz$ as $t\to+\fz$.
For any stable solution $u\in W^{1,2}_0(\Omega)$   to \eqref{st-eq} with boundary $u=0$ on $\partial \Omega$,
 we have \eqref{tho-3} when $n=10$, and \eqref{tho-4}  when $n\ge11$.
\end{cor}


\begin{rem}\label{re}\rm
(i) While  writing this paper, via personal communication
we learn that Figalli and Mayboroda have independently
proved  \eqref{tho-1} in Theorem 1.2 with $n=10$ via a similar argument.

(ii) In Theorem \ref{th} and Corollary \ref{cor}  we only consider  bounded smooth  convex domains so as to avoid  technical discussions on the  boundary estimate. We believe that after suitable modifications, it is possible to relax this assumption  to  bounded domains of $C^3$ class, as in \cite{cf}.
\end{rem}

As an application of Theorem~\ref{th}, we prove the following
 Liouville property for  stable solutions to the equation
\begin{equation}\label{st-eq2}
-\bdz u=f(u)\quad {\rm in}\quad \rr^n
\end{equation}
for $f\in C^{0,1}(\rn)$. 

\begin{thm}\label{lower}
Let $n\ge10$ and $0\le f\in C^{0,1}_{\loc}(\rr)$.
Suppose that $u\in C^2(\rr^n)$ is a nonconstant stable solution to \eqref{st-eq2}
in $\rr^n$.

 If $u$ is nonconstant, then
\begin{equation}\label{g-d2}
\mint{-}_{B_{4R}\setminus B_R}|u(x)|\,dx\ge\lf\{\begin{array}{ll}
c\log R \ &\forall R\ge R_0,\quad{\rm if}\quad n=10,\\
 cR^{-\frac n2+2+\sqrt{n-1} } \  &\forall R\ge R_0,\quad{\rm if}\quad n\ge 11
\end{array}\r.
\end{equation}
for some $R_0\ge 2$ and $c>0$.

In particular, if $u$ satisfies  the growth condition 
\begin{equation}\label{g-d}|u(x)|=\lf\{\begin{array}{ll}
o\left( \log|x|  \right) \quad&{\rm as}\ |x|\to+\fz,\quad   \quad{\rm when}\quad n=10;\\
 o\left( |x| ^{ -\frac n2+2+\sqrt{n-1}  }\right) \quad &{\rm as}\ |x|\to+\fz, \quad  \quad{\rm when}\quad n\ne10,
\end{array}\r.
\end{equation}
then $u$ must be a constant.
 \end{thm}

This problem has  attracted a lot of attention in the literature. First of all, for radial stable solutions,
Villegas \cite{v07} in 2007 obtained the following sharp Liouville property  based on the  monotone  property by Cabr\'e-Capella \cite{cc04}  (see also \cite{v07,d11}).

\begin{thm} [\cite{v07}] \label{rlp} Let $n\ge2$ and $f\in C^1(\rr)$. Suppose that
  $u \in C^2(\rn) $ is a  radial stable  solution to \eqref{st-eq2}.

 If $u$ is not constant,  then
\begin{equation}\label{g-d3} |u(x)|\ge \lf\{\begin{array}{lll}M\log |x| \ \quad & \mbox{whenever} \ |x|\ge r_0, &   {\rm when}\quad n=10, \\
  M|x|^{-\frac n2+\sqrt{n-1}+2} \ \ &\mbox{whenever} \  |x|\ge r_0 , \ &  {\rm when }\quad n\neq 10
\end{array}\r.
\end{equation}
 for some $M>0$ and $r_0\ge 10$.

In particular, if  $u$   satisfies the growth condition \eqref{g-d}, 
   then $u$ must be a constant.
\end{thm}

Note that for  radial stable solutions $u(x)$, the condition
 \eqref{g-d2} is equivalent to  \eqref{g-d3}.
Indeed, by \cite{v07},   $u(r)=u(re_1)$ is always monotone, and hence
 $$\min\{|u(4r)|,|u(r)|\}\le
\mint{-}_{B_{4r}\setminus B_r}|u(x)|\,dx\le \max\{|u(4r)|,|u(r)|\} \quad\forall r>0,$$
which implies the equivalence between   \eqref{g-d2}  and \eqref{g-d3}.



Let $\bz_n= -\frac {n}2+2+\sqrt{n-1}$. Then $\bz_n <0$ when $n\ge 11$ and $\bz_n>0$ when $ n\le 9$. The sharpness of Theorem \ref{rlp} (and also Theorem \ref{lower}) is  demonstrated in the following sense by  Villegas \cite{v07}
  (with a slight modification at $n=10$). 
 \begin{enumerate}
\item[(i)]  When $n \ne 10$,   the radial smooth function
 $ (1+|x|^2) ^{\frac  { \bz_n}2  }  $
 is a stable  solution to  the equation $-\bdz u=f_{\bz_n}(u)$  in $\rr^n$,
 where   when $n\ge11$,
\begin{align*}
&f_{\bz_n}(s):=\lf\{\begin{array}{ll}
0,\quad  &\quad{\rm if}\quad s\le0,\\
\bz_n(\bz_n-2)s^{1-
 \frac{4}{\bz_n}}-\bz_n(\bz_n+n-2)s^{1-\frac{2}{\bz_n}}, \quad &\quad{\rm if}\quad s>0
\end{array}\r.
\end{align*}
and  when $n\le 9$,\begin{align*}
 f_{\bz_n}(s):=\lf\{\begin{array}{ll}
\bz_n(\bz_n-2)s^{1-
 \frac{4}{\bz_n}}-\bz_n(\bz_n+n-2)s^{1-\frac{2}{\bz_n}},\quad  &\quad{\rm if}\quad s\ge1,\\
-(\bz_n-2)(n+2)(s-1)-n\bz_n, \quad &\quad{\rm if}\quad s<1.
\end{array}\r.
\end{align*} See  \cite[Example 3.1]{v07}  for details.
Note that, when $n\ge 11$, by $\bz_n <0$  and $\bz_n+n-2>0$, we have $f_{\bz_n}\ge0$ in $\rr$;
 while when $n\le 9$, $f_{\bz_n}$ changes sign in $\rr$.

 \item[(ii)] 
 When $n=10$, the radial smooth function
$ -\frac12\log(1+|x|^2)$ is a stable solution to the equation
$-\Delta u= f(u)$ in   $\rn$, where $f(s)=(n-2)e^{2s}+2e^{4s}\ge 0$ in $\rr$.  This is  a slight modification of the \cite[Example 3.1]{v07} with $n=10$. See the appendix for details.

  \end{enumerate}

For general (nonradial) stable solutions $u\in C^2(\rn)$ to $-\Delta u=f(u)$ in $\rn$,  it is then natural to ask if   certain Liouville properties similar to Theorem \ref{rlp} hold. 
Namely,  when $f$ satisfies certain regularity assumption, 
\begin{itemize}
\item if $u$  satisfies \eqref{g-d}, then  is it necessary that $u$ is a constant?

\item  if $u$ is nonconstant,  is it
possible to give some sharp lower bound for $|u|$ toward $\infty$?
\end{itemize}

Suppose  that $0\le f\in C^1(\rr)$ and $u\in C^2(\rn)$ is a stable solution to \eqref{st-eq2}.
When $n\le 4$, Dupaigne-Farina \cite{df} proved that if $|u|$ is bounded, then $u$ must be a constant.
Recently, with  the aid of Cabr\'e et al \cite{cf}, Dupaigne-Farina \cite{df20}  showed
that  if $n\le 9$  and $u(x)\ge -C[1+\log |x|]^{\gz}$ for some $\gz\ge 1$ and $C>0$, or if $n=10$ and $u\ge -C$ for some constant $C>0$,  then $u$ must be a constant. 
When $n\ge 10$, our result Theorem~\ref{lower} finally answers the two questions above. 


%




\subsection{Ideas of the proofs}

We sketch the ideas   to prove Theorem \ref{th} and Theorem \ref{lower}.
All  of them  heavily rely on the the following decay estimate on the Dirichlet energy.

\begin{lem}\label{key-le1}
Let  $n\ge10$ and $f\in C^{0,1}(\rr)$.
For any $y\in\rn$ and $t>0$, if
  $u\in C^2(B_{2t}(y))$   is a stable solution  to the
equation \eqref{st-eq} in $B_{2t}(y)$, one has
\begin{align}\label{mor1}
\left(\frac rt\right)^{-2(1+\sqrt{n-1})}\int_{B_{r}(y)}|Du|^2\,dx\le C(n )\int_{B_{t}(y)\setminus B_{  t/2}(y) }|Du|^2\,dx,\quad
\forall r\le  \frac  t2.
\end{align}
\end{lem}
\noindent See Section 2 for the proof of Lemma \ref{key-le1}; the key point is that we  take  a suitable test function in a celebrated lemma  of Cabr\'e et al \cite{cf} (see Lemma \ref{key-lin} below).

We also recall the following lemma, which was essentially  established in
   \cite{cf}, see Lemma A.2 and Proposition 2.5 with its proof therein. For the convenience of the reader, we give a sketch of the proof at the beginning of Section 3. 
\begin{lem}\label{L-bou}
 Let  $0\le f\in C^{0,1}(\rr)$. For any stable solution $u\in C^2(B_{2t}(y))$    to the
equation \eqref{st-eq} in $B_{2t}(y)$, one has
\begin{align}\label{w-4}
\left(\int_{B_{  t/2}(y)}|Du|^2\,dx\right)^{ 1/2}\le  C(n )t^{-n/2}\int_{B_{ t}(y)}|Du|\,dx
\end{align}
 and
\begin{align}\label{w-5}
\int_{B_{  t/2}(y)}|Du|\,dx
\le C(n )t^{-1} \int_{B_t(y)}|u|\,dx.
\end{align}
\end{lem}
Applying Lemma \ref{key-le1}, Lemma \ref{L-bou} and some known boundary estimate, we are able
to prove Theorem 1.2 and Corollary 1.3. This is clarified in Section 3.  

In order to prove Theorem \ref{lower},
an auxiliary and crucial  proposition is
 shown in Section 4, which is   especially applied in the case $n=10$. 
 
\begin{prop}\label{k-po}
Let $n\ge3$.
Suppose that    $u\in W^{1,1}_\loc(\rn)$   is superharmonic, that is, $-\Delta u\ge 0$ in $\rn$ in distributional sense.
 For any $0<r<R<\infty$ we have
\begin{align}\label{w-6}\int_{B_R\setminus B_r
} |Du||x|^{-n+1}\, dx \le C(n)  \mint-_{B_{  r/2}\setminus B_{  r/4}}|u  |\,dz+C(n)
\mint-_{B_{4R}\setminus B_{2R}}|u |\,dz.\end{align}
\end{prop}
The main idea of showing Proposition~\ref{k-po} goes as follows. First, it is known that
$$ \mbox{$Du_\dz(x)=D\Delta^{-1}[\Delta (u_\dz\eta)](x)$ for $x\in   B_R\setminus B_r$},$$  where $u_\dz$ is a standard smooth mollification of $u$ and $\eta$ is a suitable cut-off function. Next, thanks to  the key fact $-\Delta u_\dz\ge 0$, via some subtle kernel estimate and  integration by parts, we are able to
prove \eqref{w-6} for $ u_\dz$, and then a standard approximation gives  \eqref{w-6} as desired.

Theorem \ref{lower} is eventually proved in Section 5.
The case  $n\ge 11$ is relatively simpler. 
In fact, by Lemma \ref{key-le1} and  Lemma \ref{L-bou}, one can build  up the following
\begin{align*}
  r^{-(1+\sqrt{n-1})}\left(\int_{B_{r} }|Du|^2\,dx\right)^{1/2}&\le C(n )  R^{ \frac n2-2-\sqrt{n-1} }
 \mint-_{B_{3R}\setminus B_{ {3R}/4 }}| u| \,dx \quad\forall 0<r< < R<\fz 
\end{align*}
for stable solutions, which allows us to conclude Theorem \ref{lower} for $n\ge 11$. 

As for the case when $n=10$, we
first  employ   Lemma \ref{key-le1}  and repeat   Lemma \ref{L-bou} to get
\begin{align*}
  r^{-(1+\sqrt{n-1})}\left(\int_{B_{r} }|Du|^2\,dx\right)^{1/2}
&\le C(n) \frac1{\log R}
 \int_{B_{R^2}\setminus B_{4}}|D u||x|^{-n+1} \,dx\quad  \mbox{  $\forall 0<r<<R<\fz$,}
\end{align*}
 which, with the aid of Proposition~\ref{k-po},  is then bounded from above by
$$C(n)\frac 1{\log R} \left(
 \mint-_{ B_2\setminus B_1}|u(z)|\,dz+   \mint-_{B_{2R^2}\setminus B_{R^2} }|u(z)|\,dz\right).$$
From this we conclude Theorem \ref{lower} when $n=10$.

\section{Proof of Lemma \ref{key-le1}}

Towards Lemma \ref{key-le1} we recall the following  apriori bound by Cabr\'e et al \cite[Lemma 2.1]{cf}, which is obtained
by taking test function $(x\cdot Du)\eta$ in the stability condition \eqref{st-in}.

\begin{lem}\label{key-lin}
Let $u\in C^2(B_1)$ be a stable solution to equation \eqref{st-eq} in $B_1$,
with $f\in C^{0,1}(\rr)$. Then for all cut-off function $\eta\in C^{0,1}_c(B_1)$,
\begin{align}\label{key-in}
&\int_{B_1}|x\cdot Du|^2|D\eta|^2\,dx\nonumber\\
&\ge (n-2)\int_{B_1}|Du|^2\eta^2\,dx
+2\int_{B_1}|Du|^2(x\cdot D\eta)\eta\,dx
-4\int_{B_1}(x\cdot Du)(Du\cdot D\eta)\eta\,dx.
\end{align}
\end{lem}

For convenience, for any $0<r<t<\fz$ and $y\in\rn$,   write the annual  $A_{r,t}(y):=B_t(y)\setminus \overline{B_r(y)}$; for simple  write $A_{r,t}=A_{r,t}(0)$.

\begin{proof}[Proof of Lemma \ref{key-le1}]
It suffices to prove the following
\begin{align}\label{mor}
\left(\frac rt\right)^{-2(1+\sqrt{n-1})}\int_{B_{r}(y)}|Du|^2\,dx\le C(n )\int_{A_{r,t}(y)}|Du|^2\,dx,\quad
\forall r\le  \frac t2.
\end{align}
Indeed, applying \eqref{mor} to $ \frac t2$ and $t$, one has
\begin{align} \label{w-1}
\left(\frac 12\right)^{-2(1+\sqrt{n-1})}\int_{B_{\frac t2}(y)}|Du|^2\,dx\le C(n )\int_{A_{\frac t2,t}(y)}|Du|^2\,dx.
\end{align}
If  $
\frac t4\le r<\frac t2$, by $B_r(y)\subset B_{\frac  t2}(y)$  and $\frac 14\le \frac r t\le 1$, \eqref{w-1} gives
\begin{align} \label{w-2}
\left(\frac rt\right)^{-2(1+\sqrt{n-1})}\int_{B_{r}(y)}|Du|^2\,dx\le C(n )\int_{A_{\frac t2,t}(y)}|Du|^2\,dx.
\end{align}
If $0<r< \frac t4$, applying \eqref{mor} to $ r$ and $\frac t2$,and noting  $A_{r,\frac t2}\subset B_{\frac t2}$ one gets
\begin{align*} 
\left(\frac r{t/2}\right)^{-2(1+\sqrt{n-1})}\int_{B_{r}(y)}|Du|^2\,dx\le C(n )\int_{A_{ r,\frac t2}(y)}|Du|^2\,dx\le C(n )\int_{ B_{\frac t2}(y)}|Du|^2\,dx,
\end{align*}
which together with   \eqref{w-1} yields
\begin{align*}
\left(\frac r{t}\right)^{-2(1+\sqrt{n-1})}\int_{B_{r}(y)}|Du|^2\,dx\le C(n )\int_{A_{\frac t2,t}(y)}|Du|^2\,dx.
\end{align*}
From  this and \eqref{w-2} we conclude \eqref{mor1}.

To prove \eqref{mor}, without loss of generality we may assume that $t=1$ and $y=0$.
Indeed, if $u(x)$ is a stable solution to $-\Delta u=f(u)$ in $B_{2t}(y)$, then $v(x)=u(t x+y)$ is the
 stable solution to $-\Delta v=t^2 f(v)$    in $B_2$. Note that up to a change of variable $u$ satisfies \eqref{mor} if and only if
$v$ satisfies \eqref{mor} with $t=1$ and $y=0$.

Write $a=
2(1+\sqrt{n-1})$.
Let $r\in (0,\frac 12]$ be fixed and set
\begin{align}\label{test}
\eta=\left\{
\begin{aligned}
&r^{-\frac{a}{2}}&\quad{\rm if}\ 0\le |x|\le r\\
&|x|^{-\frac{a}{2}}\phi&\quad{\rm if}\ r<|x|\le 1,
\end{aligned}
\right.
\end{align}
where $\phi\in C^\fz_c(B_1)$ satisfies
\begin{equation}\label{test-1}
\phi=1\quad {\rm in}\ B_{3/4}\quad \mbox{and}\quad |D\phi|\le 5\chi_{B_1\setminus B_{3/4}}.
\end{equation}
Clearly, $\eta\in C^{0,1}_c(B_1)$. Since $\eta=r^{-\frac{a}{2}}$ in $B_r$ and hence $D\eta=0$ in $B_r$,
submitting $\eta$ in inequality \eqref{key-in} one has
\begin{align}\label{e2.y3}
&\int_{A_{r,1}}|x\cdot Du|^2|D\eta|^2\,dx\nonumber\\
&\ge (n-2)r^{-a}\int_{B_r}|Du|^2\,dx+(n-2)\int_{A_{r,1}}|Du|^2\eta^2\,dx\nonumber\\
&\quad+2\int_{A_{r,1}}|Du|^2(x\cdot D\eta)\eta\,dx
-4\int_{A_{r,1}}(x\cdot Du)(Du\cdot D\eta)\eta\,dx.
\end{align}
Noting that
$$D\eta=-\frac{a}{2}|x|^{-\frac{a}{2}-2}x\phi+|x|^{-\frac{a}{2}}D\phi
\quad{\rm in}\ A_{r,1},$$
one has
\begin{align}\label{e2.y1}
&
2\int_{A_{r,1}}|Du|^2(x\cdot D\eta) \eta\,dx-4\int_{A_{r,1}}(x\cdot Du)(Du\cdot D\eta) \eta\,dx\nonumber\\
&=-a\int_{A_{r,1}} |Du|^2|x|^{-a}\phi^2\,dx +
2\int_{A_{r,1}}|Du|^2(x\cdot D\phi) \phi|x|^{-a}\,dx \nonumber\\
&\quad+
2a\int_{A_{r,1}} (x\cdot Du)^2|x|^{-a-2}\phi^2\,dx
-4\int_{A_{r,1}}(x\cdot Du)(Du\cdot D\phi) \phi|x|^{-a}\,dx.
\end{align}

Moreover, by
$$|D\eta|^2=\frac{a^2}{4}|x|^{-a-2}\phi^2-2a|x|^{-a-2}(x\cdot D\phi) \phi+|x|^{-a}|D\phi|^2, $$
one can write

\begin{align}\label{e2.y2}
\int_{A_{r,1}}(Du\cdot x)^2|D\eta|^2\,dx&=\frac{a^2}{4}\int_{A_{r,1}} (Du\cdot x)^2|x|^{-a-2}\phi^2\,dx
+\int_{A_{r,1}} (Du\cdot x)^2|x|^{-a}|D\phi|^2\,dx\nonumber\\
&\quad-a\int_{A_{r,1}}(Du\cdot x)^2|x|^{-a-2}(x\cdot D\phi) \phi\,dx.
\end{align}

Using \eqref{e2.y1} for the left hand side of \eqref{e2.y3} and \eqref{e2.y2} for the last two terms in the right  hand side of \eqref{e2.y3},
and then moving all terms including $D\phi$ to the left hand side and all other terms to the right hand side, we have

\begin{align}\label{eq-2}
&\int_{A_{r,1}}|x\cdot Du|^2|D\phi|^2|x|^{-a}\,dx
-2\int_{A_{r,1}}|Du|^2(x\cdot D\phi)\phi|x|^{-a}\,dx\nonumber\\
&\quad+4\int_{A_{r,1}}(x\cdot Du)(Du\cdot D\phi)\phi
|x|^{-a}\,dx-a\int_{A_{r,1}}|x|^{-a-2}(x\cdot Du)^2\phi(x\cdot D\phi)\,dx\nonumber\\
&\ge (n-2)r^{-a}\int_{B_r}|Du|^2\,dx+(n-2)\int_{A_{r,1}}|Du|^2|x|^{-a}\phi^2\,dx\nonumber\\
&\quad-a\int_{A_{r,1}} |Du|^2|x|^{-a}\phi^2\,dx + 2a\int_{A_{r,1}} (x\cdot Du)^2|x|^{-a-2}\phi^2\,dx\nonumber\\
&\quad-\frac{a^2}{4}\int_{A_{r,1}} (Du\cdot x)^2|x|^{-a-2}\phi^2\,dx\nonumber\\
&=
(n-2)r^{-a}\int_{B_r}|Du|^2\,dx\nonumber\\
&\quad+ \int_{A_{r,1}}\left\{(n-2-a)|Du|^2+
\left(2a-\frac{a^2}{4}\right)(Du\cdot x)^2|x|^{-2}
\right\}|x|^{-a}\phi^2\,dx.
\end{align}

Note that by $|D\phi|=0$ in $B_{3/4}$ and $|D\phi|\le 5$ in $B_1$ as in \eqref{test-1} and $a> 2$,
\begin{align}\label{e2.y4}
&\int_{A_{r,1}}|x\cdot Du|^2|D\phi|^2|x|^{-a}\,dx
-2\int_{A_{r,1}}|Du|^2(x\cdot D\phi)\phi|x|^{-a}\,dx\nonumber\\
&\quad+4\int_{A_{r,1}}(x\cdot Du)(Du\cdot D\phi)\phi
|x|^{-a}\,dx-a\int_{A_{r,1}}|x|^{-a-2}(x\cdot Du)^2\phi(x\cdot D\phi)\,dx\nonumber\\
&\le  C(n)\int_{A_{\frac34,1}}|Du|^2 \,dx.
\end{align}
Additionally, note that
  $n\ge 10$ implies  $a=2(1+\sqrt{n-1})\ge 8$,  and hence
$$2a-\frac{a^2}{4}=\frac a4(8-a)\le 0.$$
By $|x|^{-1}|x\cdot Du|\le |Du|$ in $A_{r,1}$ we have
$$(n-2-a)|Du|^2+
\left(2a-\frac{a^2}{4}\right)(Du\cdot x)^2|x|^{-2}
\ge \left(n-2+a-\frac{a^2}{4}\right)|Du|^2.$$
Since
$$
n-2+a-\frac{a^2}{4}= -(\frac a2-[1-\sqrt{n-1}])(\frac a2-[1+\sqrt {n-1}])=0, $$
 we have
\begin{align}\label{e2.y5}  (n-2-a)|Du|^2+
\left(2a-\frac{a^2}{4}\right)(Du\cdot x)^2|x|^{-2} \ge 0 \quad\mbox{in $A_{r,1}$},\end{align}
which  means that  the last term in the right hand side of \eqref{eq-2} is  nonnegative.
From this, together with \eqref{eq-2} and \eqref{e2.y4} we conclude
  \eqref{mor}.   The proof is complete.
\end{proof}

\begin{rem}\rm  Recall that in   \cite{cf}, Cabr\'e et al used
 the test funciton   $\eta=|x|^{-\frac a2}\xi$ with $\xi\in C^\fz_c(B_1)$, which was not enough to get \eqref{mor}.
\end{rem}

\section{Proofs of  Theorem 1.2 and Corollary 1.3}
In this section we   prove Theorem \ref{th} and Corollary \ref{cor}.
First, we show a sketch of the proof of Lemma \ref{L-bou}.

\begin{proof} [Proof of Lemma \ref{L-bou}]
 Up to considering
  $v(x)=u(t x+y)$ as did in the proof of Lemma 1.4, we may assume that $t=1$ and $y=0$.
The inequality \eqref{w-5} is given by \cite[Lemma A.2]{cf}.
The inequality \eqref{w-4}
reads as $\|Du\|_{L^2(B_{{1/2}})}\le C(n)\|Du\|_{L^1(B_{1})}$, and will follow  from the proof of \cite[Proposition 2.5]{cf}. Indeed,
in  \cite[Proposition 2.5]{cf}, Cabre et al  proved that
\begin{equation}\label{sxx} \|Du\|_{L^2(B_{{1/2}})} \le C(n) \|u \|_{L^1(B_{1}) }. \end{equation}
In their proof, first they  obtained a bound of $\|Du\|_{L^2(B_{{1/2}})}$ via $\|Du\|_{L^1(B_{1/2})}$  and also some other small terms.  Next, they used
$\|Du\|_{L^1(B_{{1/2}})}\le C(n)\|u \|_{L^1(B_1) }$.  Finally, via  an iteration argument, they got \eqref{sxx}. If   we  directly apply the iteration argument
 without using $\|Du\|_{L^1(B_{{1/2}})}\le C(n)\|u \|_{L^1(B_1) }$,  one gets
 $\|Du\|_{L^2(B_{{1/2}})}\le C(n)\|Du\|_{L^1(B_{1})}$.
\end{proof}


Recall that $u_{E}=\mint{-}_Eu\,dx$ denotes the integral average of $u$ on a measurable set $E$.
The interior regularity \eqref{tho-1}\& \eqref{tho-2} in Theorem \ref{th} is
a consequence of  Lemma \ref{key-le1} and \eqref{w-4}, together with some standard embedding argument.

\begin{proof}[Proofs of \eqref{tho-1} and \eqref{tho-2} in Theorem \ref{th}]

   Let $u\in C^2(B_2)$ be stable solution
to equation \eqref{st-eq}.
 Write $\bz=n-2-2\sqrt{n-1}$. 
For any $  y\in B_{1/2}$ if $ r>1/8$, by Lemma \ref{L-bou} we have
$$ r^{\bz-n}\int_{B_r(y)\cap B_{1/2}}|Du|^2\,dx \le  C(n)\mint{-}_{  B_{1/2}}|Du|^2\,dx \le C(n)\|u\|^2_{L^1(B_1)}$$
and if $0<r<1/8$, by Lemma \ref{key-le1} and Lemma \ref{L-bou} again we have
$$ r^{\bz-n}\int_{B_r(y)\cap B_{1/2}}|Du|^2\,dx \le  r^{\bz}\mint{-}_{  B_r(y)}|Du|^2\,dx \le  C(n)\mint{-}_{  B_{1/4}(y)}|Du|^2\,dx
 \le  C(n)
\|u\|^2_{L^1(B_1)}.$$
This means that $Du\in M^{2,\beta}(B_{1/2})$ with $\|Du\|_{M^{2,\beta}(B_{1/2})}\le    C(n)
\|u\|_{L^1(B_1)}$.

If $n=10$, then $\bz=2$ and $ {2\bz}/{(\bz-2)}=\fz$. Thanks to Sobolev-Poincar\'e inequality one can easily check that $Du\in M^{2,\beta}(B_{1/2})$ implies   $u\in BMO(B_{1/2}) $ with a norm bound
$\|u\|_{BMO(B_{1/2})}\le C(n)\|Du\|_{M^{2,\beta}(B_{1/2})} $.
If $n\ge11$, then $ p_n= {2\bz}/{(\bz-2)}<\fz$ and $\bz=4+2/(p_n-2)$.
By the embedding result in \cite{a75} and also \cite[Section 4]{cf21},  $Du\in M^{2,\beta}(B_{1/2})$ implies   $u\in M^{2\bz/(\bz-2),\beta}(B_{1/2})$  with its norm bound
$\|u\|_{M^{p_n,\bz}(B_{1/2})}\le C(n)\|Du\|_{M^{2,\beta}(B_{1/2})} $.
This proves \eqref{tho-1} and \eqref{tho-2}.
%
%
\end{proof}

To prove the global regularity \eqref{tho-3} and \eqref{tho-4} in Theorem 1.2,
we need the following
 a priori $L^\fz$-bound  in a neighborhood of
$\partial\Omega$ for $C^2$   solution   when $\Omega$ is a bounded smooth convex domain, see \cite[Proposition 3.2]{c10} and
\cite{cl93,fl82,gnn}. 
For $\rho>0$ we write  $$\Omega_{\rho}:=\{x\in \Omega:{\rm dist}(x,\partial \Omega)<\rho\}.$$
\begin{lem}\label{mov}
Suppose that  that $f\in C^{0,1}(\rr)$ is nonnegative and $\Omega$ is a  smooth convex domain
in $\rr^n$. There
exist positive constants $\rho$ and $\gz$ depending only on
the domain $\Omega$ such that for any positive solution $u\in C^2(\Omega)\cap C^0(\overline \Omega)$ to \eqref{st-eq} one has
\begin{equation}\label{bound}
\|u\|_{L^\fz(\Omega_{\rho})}\le \frac{1}{\gz}\|u\|_{L^1(\Omega)}.
\end{equation}
\end{lem}

Note that, as $f\ge 0$,   the maximal principle  shows that any  solution $u\in C^2(\Omega)\cap C^0(\overline \Omega)$ to \eqref{st-eq}  with   $0$ boundary is always nonnegative;
and the strong maximal principle further shows that $u$ is always positive in the domain $\Omega$.


\begin{proof}[Proofs of \eqref{tho-3} and \eqref{tho-4} in Theorem 1.2]

Let  $\bz=n-2-2\sqrt{n-1}$ and let $\rho,\gz$   be as in Lemma \ref{mov}.  We first consider the case $n\ge 11$.
For any $y\in\overline {\Omega } $ and $r>0$, write
\begin{align*}
r^{\bz-n} \int_{\Omega \cap B_r(y)}|u|^{p_n}\,dx&= r^{\bz-n} \int_{\Omega_{\rho} \cap B_r(y)}|u|^{p_n}\,dx
+ r^{\bz-n} \int_{(\Omega\setminus \Omega_{\rho})\cap B_r(y)}|u|^{p_n}\,dx\\
&:=\Phi_1(y,r)+\Phi_2(y,r).\end{align*}
To see \eqref{tho-3}, obviously,
we only need to prove
$\Phi_1(y,r) \le    C(n,\Omega)  \|u\|_{L^1(\Omega)} ^{p_n}$
and $\Phi_2(y,r)  \le C(n,\rho,\Omega) \|u\|^{p_n}_{L^1(\Omega)}$
for any $y\in \Omega   $ and $r>0$.

Note that   $$r^{\bz-n}  |\Omega_\rho\cap B_r(y) |\le
C(n) r^{\bz-n} \le C(n) \quad\mbox{when $r<1$ and $\le |\Omega_\rho|$ when $r>1$,}$$
by $2<\bz<n$ and Lemma \ref{mov}, we have
\begin{align*}
\Phi_1(y,r) \le r^{\bz-n} |\Omega_\rho\cap B_r(y) | \|u\|_{L^\fz(\Omega_\rho)}^{p_n} \le    C(n,\Omega)  \|u\|_{L^1(\Omega)} ^{p_n}.
\end{align*}

Next, to get  $\Phi_2(y,r)  \le C(n,\rho,\Omega) \|u\|^{p_n}_{L^1(\Omega)}$
for any $y\in \Omega   $ and $r>0$, we  only need to consider
 $y\in \Omega\setminus   \Omega_{\rho} $ and $0<r<\rho/8$.
Indeed, for $y\in \Omega_\rho$, if $r<\dist(y, \Omega \setminus \Omega_{\rho} )$, then $\Phi_2(y,r)=0$,
and if $r\ge \Omega \setminus \Omega_{\rho}$, then
$\Phi_2(y,r)\le C(n)\Phi_2(\bar y,2r)$, where   $\bar y$ is the closest point in $\Omega\setminus\Omega_\rho $ and $B(y,r)\subset B(\bar y,2r)$.
Moreover for any $y\in \Omega\setminus   \Omega_{\rho} $ and $r\ge\rho/8$,
$$\Phi_2(y,r) \le  \rho^{\bz-n} \int_{\Omega \setminus \Omega_\rho }|u|^{p_n}\,dx\le \sum_{i=1}^N
\rho^{\bz-n} \int_{\Omega \setminus \Omega_\rho \cap B_{\rho/9}(x_i)}|u|^{p_n}\,dx=  \sum_{i=1}^N \Phi(x_i,\rho/9),  $$
 where $\{B(x_i,\rho/9)\}_{i=1}^N$  is a cover of the compact set $\Omega \setminus \Omega_\rho $,   $\{x_i\}_{i=1}^N\subset \Omega\setminus \Omega_\rho$ and  $N$ depending on $ \Omega$ and $\rho$.

On the other hand, for any  $y \in { \Omega\backslash \Omega_{ \rho } }$ and $0<r<\rho/8$, since $u$ is stale  solution in $B_{\rho}(y)\subset\Omega$, by \eqref{tho-2} with a scaling argument we have
$u\in M^{p_n,\bz} (B_{\rho/8}(y))$ with
$\|u\|_{M^{p_n,\bz} (B_{\rho/8}(y))}\le C(n,\rho)\|u\|_{L^1(B_{\rho/2}(y))}$, in particular
$$\Phi_2(y,r)\le r^{\bz}\mint{-}_{B_r(y)}|u|^{p_n}\,dx\le
 C(n,\rho)
\|u\|^{p_n} _{L^1(\Omega)}$$
as desired. This proves \eqref{tho-4}.

In the case   $n=10$,
for any $y\in\Omega$, if $ r>\frac19\rho$, we have
$$ r^{-n}\int_{\Omega\cap B_r(y)}|u|\,dx\le C(n,\rho) \|u\|_{L^1(\Omega)}.$$
Below we assume that $ 0<r<\frac19\rho$.
If  $y\in\Omega\setminus\Omega_{8\rho/9}$, we have $\rho <\frac98\dist(y,\partial\Omega)$. Since $0<r<\frac18\dist(y,\partial \Omega)$ and
  $u$ is a stale  solution in $B_{\dist(y,\partial\Omega)}(y)\subset\Omega$, by \eqref{tho-1} with a scaling we have
$$ \mint{-}_{B_r(y)}|u-u_{B_r(y)}|\,dx \le C(n,\rho)\|u\|_{L^1(B_{\dist(y,\partial\Omega)}(y))}\le C(n,\rho)\|u\|_{L^1(\Omega)}.$$
For  $y\in  \Omega_{8\rho/9}$,  noting $0<r<  \frac19\rho\le \dist(y,\partial\Omega_\rho)$, one has  $\Omega\cap B_r(y)\subset\Omega\setminus \Omega_\rho$. Thus
$$r^{-n}\int_{\Omega \cap B_r(y)}|u|\,dx= r^{-n}\int_{\Omega_\rho\cap B_r(y)}\left|u \right|\,dx\le C(n,\rho)\|u\|_{L^1(\Omega)}.$$
Combining these estimates,
 we obtain \eqref{tho-3}.
\end{proof}
We finally prove Corollary \ref{cor}.

\begin{proof}[Proof of Corollary \ref{cor}]
  Let $u\in W^{1,2}_0(\Omega)$ be
a stable solution to \eqref{st-eq} with $0$ boundary. By \cite[Corollary 3.2.1]{d11}(see also the proof in \cite[Theorem 4.1]{cf}
 and \cite[Theorem 5]{df}), there is a nonnegative, nondecreasing sequence $(f_k)$  of convex functions
in $C^{1}(\rr)$ such that $f_k\to f$ pointwise in $[0,\fz)$  and
a nondecreasing sequence $(u_k)$ in $C^2(\overline \Omega)\cap W^{1,2}_0(\Omega)$ such that
$u_k$ is a weak stable solution to
\begin{align}\label{ap-di}
-\bdz u_k=f_k(u_k)\quad{\rm in}\quad \Omega;\ u_k=0\quad{\rm on}\quad \partial\Omega.
\end{align}
%
and
\begin{equation*}\label{th-1}
u_k\to u\quad {\rm in}\quad W^{1,2}(\Omega)\quad{\rm as}\quad k\to+\fz.
\end{equation*}

If $n=10$, applying \eqref{tho-3} to $u_k$ one has
\begin{align*}
\mint{-}_{\Omega \cap B_r(y)}\left|u_k(x)-\mint{-}_{\Omega \cap B_r(y)}u_k\,dz\right|\,dx
\le \|u_k\|_{{\rm BMO}(\Omega)}\le C(n,\Omega)\int_{\Omega}|u_k|\,dx,\quad
\forall r>0,\ \forall y\in \overline \Omega.
\end{align*}
By $u_k\to u$ in $W^{1,2}(\Omega)$ as $k\to +\fz$ we conclude that $\|u\|_{BMO(\Omega)}\le C(n)\|u\|_{L^1(\Omega)}$ as desired.

If $n\ge 11$, applying \eqref{tho-4} to $u_k$ we have
\begin{align}\label{cor-1}
r^{\bz-n}\int_{\Omega\cap B_r(y)}|u_k|^{p_n}\,dx\le C(n,\Omega,\rho)
(\|u_k\|_{L^1(\Omega)})^{p_n},\quad \forall y\in \overline \Omega,\quad
\forall r>0,
\end{align}
where $\bz=\frac{2p_n}{p_n-2}\in (0,n)$.
By $u_k\to u$ in $W^{1,2}(\Omega)$ as $k\to+\fz$ we deduce that
$u_k\in L^{p_n}(\Omega)$ uniformly in $k\ge 0$, and  hence,
$u_k\to u$ weakly in $L^{p_n}(\Omega)$.  Thus, letting $k\to +\fz$ in \eqref{cor-1} we conclude
 $\|u\|_{M^{p_n,\beta}(\Omega)}\le C(n)\|u\|_{L^1(\Omega)}$ as desired.
\end{proof}

\section{Proof of Proposition \ref{k-po} }

Let $0<r<R<\fz$. Let $\eta\in C^\fz_c(A_{\frac r4, 4R})$ satisfy that
\begin{align}
0\le \eta\le 1 \quad \mbox{in}\ A_{\frac r4,4R}\ \quad \mbox{and}\quad
\eta=1 \ \mbox{in}\ A_{\frac r2,2R}, \label{eta0}\\
  |D\eta|^2+
 |D^2\eta|\le  \frac C{r^2 }\ \mbox{  in } \  A_{\frac r4, \frac r2} \
  \mbox{and}\   |D\eta|^2 +  |D^2\eta|\le  \frac C{R^2}\ \mbox{in $A_{2R, 4R}$,}\label{eta00}
\end{align}
where $C>0$ is  a universal constant.

Let $ u_\delta=u\ast\phi_\dz $ for $\dz>0$, where $\phi_\dz$ is the standard smooth mollifier and supported in $B(0,\dz)$.
Recall that $u\in W^{1,\,1}_\loc(\mathbb R^n)$ and $u_\delta\to u$ in $ W^{1,\,1}_\loc(\mathbb R^n)$.
 Since $-\Delta u\ge 0$ is a locally finite measure,  we have $- \Delta u_\dz=(-\Delta u)\ast\phi_\dz \ge 0$ everywhere.
By $ u_\dz\eta \in C^\fz_c(\rn)$, one has  $$u_\dz\eta(x) =\Delta^{-1} [\Delta (u_\dz\eta)](x)=c(n)\int_{\rn}\frac1{|x-y|^{n-2}}\Delta (u_\dz\eta)(y)\,dy\quad\forall x\in\rn$$
and hence
$$D (u_\dz\eta)(x) =D\Delta^{-1} [\Delta (u_\dz\eta)](x)=c(n)(2-n)\int_{\rn}\frac{x-y}{|x-y|^{ n}}  \Delta (u_\dz\eta)(y)\,dy\quad\forall x\in\rn.$$
Noting $$\Delta (u_\dz\eta)(y)=\Delta  u_\dz(y) \eta (y) +u_\dz(y) \eta (y)+
D  u_\dz(y)\cdot D\eta (y)  ,$$
for $0<\dz<<r/8$ we write
\begin{align*}
\int_{A_{r,R}}  {|Du_\dz |}{|x|^{-n+1}}\, dx=&  \int_{A(r,\,R)} |D(u_\dz \eta)|{|x|^{ -n+1}}\, dx\\
= & \int_{A_{r,R}} \left|\int_{\mathbb R^n}\frac{x-y}{|x-y|^{ n}}   \Delta (u_\dz\eta) (y)\, dy\right||x|^{-n+1}\, dx\\
\le&   C(n)\int_{A_{r,R}} \left|\int_{\mathbb R^n}\frac{x-y}{|x-y|^{ n}}   \Delta  u_\dz (y)  \eta(y)\, dy\right|{|x|^{ -n+1}}\, dx\\
&+  C(n)\int_{A_{r,R}} \left|\int_{\mathbb R^n}\frac{x-y}{|x-y|^{ n}}     u_\dz (y) \Delta  \eta(y)\, dy\right|{|x|^{ -n+1}}\, dx\\
&+  C(n)\int_{A_{r,R}} \left|\int_{\mathbb R^n}\frac{x-y}{|x-y|^{ n}}   D  u_\dz (y) \cdot D \eta(y)\, dy\right|{|x|^{ -n+1}}\, dx\\
=:&I_1+I_2+I_3.
\end{align*}

In order to control $I_1$ from above, first 
by $-\Delta  u_\dz\ge0$ and \eqref{eta0} one has
$$
I_1\le  \int_{\rn}  \int_{\rn}  |x-y|^{-n+1}|x|^{-n+1}\, dx(-\Delta  u_\dz)(y)  \eta(y) \, dy. $$
Employing the triangle inequality, for $y\in \rn$, we further get
\begin{align}\label{yn2}
   \int_{\rn}|x-y|^{-n+1}|x|^{-n+1}\,dx & \le 2^{n-1} \int_{\{|x|>2|y|\} }|x|^{-2n+2}\, dx +  2^{n-1}  \int_{\{ |x|<\frac12|y|\} }|x|^{-n+1}|y|^{-n+1}\, dx \nonumber\\
&\quad +    \int_{\{\frac12|y| \le |x|\le 2|y|\} }|x-y|^{-n+1}|y|^{-n+1}\, dx \nonumber\\
 &   \le  C(n)|y|^{-n+2}+ C(n)|y|^{-n+2}+ \int_{\{|y-x|\le 3|y|\}}|x-y|^{-n+1}|y|^{-n+1}\, dx\nonumber\\
&\le  C(n)|y|^{-n+2}.
\end{align}
This together with $-\Delta u_\dz\ge0$ again gives
$$I_1 \le C(n)    \int _{\rn }  (-\Delta u_\dz )  |y|^{-n+2}\eta(y)\, dy.$$

Via integration by parts  and using   $\eta
\in C^\fz_c(A_{\frac r4,4R})$, we have
$$
\int _{\rn }  (-\Delta u_\dz )  |y|^{-n+2}\eta(y)\, dy= \int _{A_{\frac r4,4R} }    u_\dz    [-\Delta|y|^{-n+2}\eta(y)+ D|y|^{-n+2}\cdot D\eta(y)-|y|^{-n+2}\Delta\eta(y)]
 \,dy. $$
Observing that $\Delta |y|^{n-2}=0$  in $A_{\frac r4,4R}$,  and   using \eqref{eta0}\&\eqref{eta00} we arrive at
\begin{align*}
I  _1
&\le C(n)   \int _{A_{\frac r4,4R}}  u_\dz(y)[  (2-n) |y|^{-n} y\cdot D\eta(y) - |y|^{-n+2} \Delta \eta(y)]\, dy\\
&\le C(n)  \int_{A_{\frac r4,4R}}  |u_\dz (y)| [r^{-n} \chi_{A_{\frac r4,\frac r2}}+R^{-n}\chi_{A_{2R,4R}}] \, dy \\
&\le  C(n) \mint-_{A_{\frac r4,\frac r2}}|u_\dz |\,dz+ \mint-_{A_{2R,4R}}|u_\dz |\,dz.
\end{align*}

For $I_2$,  by \eqref{yn2} and \eqref{eta0}  
\begin{align*} I_2 &\le
  \int_{\rn} \int_{A(r,\,R)} |x-y|^{-n+1} |x|^{-n+1}\, dx  | u_\dz| (y) |\Delta  \eta (y)|\, dy \\
&\le C(n)\int_{\rn} |y|^{-n+2}  | u_\dz| (y) |\Delta  \eta (y)|\,dy\\
&\le C(n)  \int_{\rn}  |u_\dz (y)| [r^{-n} \chi_{A_{\frac r4,\frac r2}}+R^{-n}\chi_{A_{2R,4R}}] \, dy \\
& \le C(n)  \mint-_{A_{\frac r4,\frac r2}}|u_\dz |\,dz+ C(n)  \mint-_{A_{2R,4R}}|u_\dz |\,dz .  \end{align*}

Now let us estimate $I_3$.
First via integration by parts one gets
\begin{align*}
& \int_{\mathbb R^n}|x-y|^{-n}(x-y) D  u_\dz(y) \cdot D_\dz \eta(y)\, dy \\
&=
 \int_{\mathbb R^n}|x-y|^{-n}(x-y)   u_\dz (y)  \Delta  \eta(y)\, dy+
 \int_{\mathbb R^n}  u_\dz (y) D[|x-y|^{-n}(x-y) ]     D \eta(y)\, dy.
 \end{align*}
Since
$$
 |D[|x-y|^{-n}(x-y) ] |\le C(n)|x-y|^{-n},$$
we   obtain
\begin{align*}
&\left|\int_{\mathbb R^n}|x-y|^{-n}(x-y) D  u_\dz(y) \cdot D  \eta(y)\, dy\right|\\
&\le C(n)\left| \int_{\mathbb R^n}|x-y|^{-n+1}    u _\dz(y)   \Delta  \eta(y) \, dy\right|+ C(n)
\int_{\mathbb R^n}   |x-y|^{-n }   |u_\dz (y)|   |D \eta(y)|\, dy.
 \end{align*}
As a consequence,
\begin{align*}
I_3\le&  C(n) I_2  +
C(n)\int_{\mathbb R^n}\int_{A_{r,R}}  |x-y|^{-n } |x|^{-n+1} \,dx |u_\dz (y)|  |D \eta(y)|\, dy=: C(n) I_2+ C(n)\wz {I_3}.
  \end{align*}

In order to estimate $\wz {I_3}$, first we note that \eqref{eta0} gives
$$\wz {I_3}\le C(n) \int_{\rn}\int_{A_{r,R}}  |x-y|^{-n } |x|^{-n+1} \,dx |u_\dz (y)| [r^{-1} \chi_{A_{\frac r4,\frac r2}}+R^{-1}\chi_{A_{2R,4R}}] \, dy $$
For any $x\in A_{r,R}$, if $ y\in A_{\frac r4,\frac r2}$  we have   $|x-y|\ge |x|/2 $
and hence $$
\int_{A_{r,R}}  |x-y|^{-n } |x|^{-n+1} \,dx\le C(n)\int_{A_{r,R}}|x|^{-2n+1}\,dx\le C(n) r^{-n+1};$$
 if  $ y\in A_{2R,4R}$, then
  $|x-y|\ge  R $  and hence
$$
\int_{A_{r,R}}  |x-y|^{-n } |x|^{-n+1} \,dx\le C(n) R^{-n} \int_{A_{r,R}}   |x|^{-n+1} \,dx \le C(n)R^{-n+1}.$$
Thus it follows that
$$\wz {I_3}\le C(n) \int_{\rn}  |u_\dz (y)| [r^{-n} \chi_{A_{\frac r4,\frac r2}}+R^{-n}\chi_{A_{2R,4R}}] \, dy\le C(n)  \mint-_{A_{\frac r4,\frac r2}}|u_\dz |\,dz+C(n)  \mint-_{A_{2R,4R}}|u_\dz |\,dz.  $$

To conclude,
$$
\int_{A_{r,R}} |Du_\dz ||x|^{-n+1}\, dx\le   C(n) \mint-_{A_{\frac r4,\frac r2}}|u_\dz |\,dz+C(n)   \mint-_{A_{2R,4R}}|u_\dz |\,dz. $$
By letting $\dz\to0$ and noting $ u_\dz\to  u$ in $W^{1,\,1}_\loc$, we conclude \eqref{w-6}.

\section{Proof of Theorem \ref{lower}}
Since $u$ satisfies  \eqref{g-d}, then $u$ does not satisfy \eqref{g-d2}.  We only need to show that
if $u$ is nonconstant, then \eqref{g-d2} holds.
Equivalently, it suffices to show that if $u$ does not satisfy \eqref{g-d2},
then $u$ is a constant.  Namely, there exists a sequence $\{R_j\}_{j\in\nn}$ towards $\infty$ so that  
\begin{align}\label{n10}
\frac 1{\log R_j  } \mint-_{ A_{   R_j    , 4R_j  }}|u(z)|\,dz\to 0\quad\mbox{ as $j\to\fz$\quad when $n=10$ } \end{align}
and
\begin{align}\label{n11} R_j^{ \frac n2 - 2 -\sqrt{n-1} } \mint-_{A_{  R_j, 4R_j}}|u(x)|\,dx\to 0 \quad\mbox{as $j\to\fz$\quad when $n\ge11$.}
\end{align}

On the other hand, given any $0<r<\fz$, applying \eqref{mor1} for any $ R> 4  r$   we have
\begin{align*}
 r ^{- (1+\sqrt{n-1})}
\left(
\int_{B_{r} }|Du|^2\,dx\right)^{1/2}\le C(n )  R ^{- (1+\sqrt{n-1})} \left(\int_{A_ {R ,2R} }|Du|^2\,dx\right)^{1/2}.
\end{align*}
Observe that  the annual  $A_{  1  ,2} $ can be covered by  $\{B_{\frac 18}(y_i)\}_{i=1}^N$ with $y_1,\cdots, y_N\in A_{  1  ,2}$  and  $N\le C(n)$.
 $$ \chi_{A_{1, 2 }}\le \sum_{i=1}^N\chi_ {B_{\frac 18}(y_i) } \le \sum_{i=1}^N\chi_ {B_{\frac 14}(y_i) } \le  C(n)\chi_{A_{\frac 34, 3}}.$$
 Below we consider the case $n\ge11$ and the case $n=10$ separately.

\medskip
\noindent{\bf Case $n\ge 11$.} For each $i$, applying  \eqref{w-4} and \eqref{w-5},   one attains
$$\left(\int_{B_{\frac R8}(Ry_i) }|Du|^2\,dx\right)^{1/2}\le C(n)
R^{-\frac {n+2}2 }   \int_{B_{\frac R4}(Ry_i) }| u| \,dx \le C(n)
R^{ \frac {n-2}2}  \mint-_{A_{\frac {3R}4,3R}}| u| \,dx.
$$
Thus by summing over all these balls,
\begin{align*}
\int_{A_ {R ,2R} }|Du|^2\,dx&\le   C(n)  R^{ n-2}
\left(\mint-_{A_{\frac {3R}4,3R}}| u| \,dx\right)^{2},
\end{align*}
and we eventually obtain
\begin{align*}
  r^{-(1+\sqrt{n-1})}\left(\int_{B_{r} }|Du|^2\,dx\right)^{1/2}&\le C(n )  R^{ \frac n2-2-\sqrt{n-1} }
 \mint-_{A_{\frac {3R}4,3R}}| u| \,dx.
\end{align*}
Taking $R= \frac {4 R_j} 3$, applying \eqref{n11} and letting $j\to\fz$,  one concludes
$$
\int_{B_{r} }|Du|^2\,dx=0 .$$
By the arbitrariness of $r>0$,  we obtain $\|Du\|_{L^2(\rn)}=0$, which implies that $u$ is a constant.

\medskip
\noindent
{\bf Case $n=10$.}
For each $i$, applying  \eqref{w-4},   one attains
$$\left(\int_{B_{\frac R8}(Ry_i) }|Du|^2\,dx\right)^{1/2}\le C(n)
R^{-\frac n2 }   \int_{B_{\frac R4}(Ry_i) }|D u| \,dx \le C(n)
R^{ \frac {n-2}2}  \int_{A_{\frac R2,4R}}|Du| |x|^{-n+1}\,dx.
$$
Thus
\begin{align}\label{iter}
\int_{A_ {R ,2R} }|Du|^2\,dx&\le   C(n)  R^{n-2 }
\left(\int_{A_{\frac R2,4R}}|D u| |x|^{-n+1}\,dx\right)^{2}.
\end{align}
We therefore obtain
\begin{align*}
  r^{-(1+\sqrt{n-1})}\left(\int_{B_{r} }|Du|^2\,dx\right)^{1/2}&\le C(n )  R^{-\frac n2-2-\sqrt{n-1} }
 \int_{A_{\frac R2,4R}}|D u||x|^{-n+1} \,dx\\
&=C(n)
 \int_{A_{\frac R2,4R}}|D u||x|^{-n+1} \,dx,
\end{align*}
where in the last identity we use $ \frac n2 - 2 -\sqrt{n-1}=5-2-3=0$.

For $R> 2^5+r>4$,  let $m$  be the largest integer so that $m\le \log_2 R-3 $.  Applying the  \eqref{iter}   to $2^j R$ with $j=1,\cdots m $, one has
\begin{align*}
  r^{-(1+\sqrt{n-1})}\left(\int_{B_{r} }|Du|^2\,dx\right)^{1/2}
&\le  C(n) \frac 1m \sum_{j=1}^m \int_{A_{\frac {2^jR}2,4(2^jR)}}|Du| |x|^{-n+1} \,dx\\
&\le  C(n) \frac 1m \int_{A_{R,2^{m+2}R}} |Du| |x|^{-n+1} \,dx\\
&\le   C(n) \frac 1{\log R} \int_{A_{4, \frac{R^2}2} }  |Du| |x|^{-n+1} \,dx.
\end{align*}
By \eqref{w-6}, one has
$$
 r^{-(1+\sqrt{n-1})}\left(\int_{B_{r} }|Du|^2\,dx\right)^{1/2}  \le C(n)\frac 1{\log R}
 \mint-_{ A_{1,2}}|u(z)|\,dz+ C(n)\frac 1{\log R^2} \mint-_{ A_{    {R^2} , 2 R^2 }}|u(z)|\,dz.$$
Taking $R=\sqrt{R_j}$ and letting $j\to\fz$,
by \eqref{n10}   one concludes
$$  \int_{B_{r} }|Du|^2\,dx=0.$$   
Then the arbitrariness of $r>0$ implies $\|Du\|_{L^2(\rn)}=0$, which further implies that $u$ is a constant.

 \section*{Appendix \ A radial stable solution when $n=10$}

\par\indent
Suppose $n=10$ in this appendix.  Villegas \cite{v07} proved that
$ \frac 12\log(1+|x|^2)$ is a stable solution to equation
 $-\bdz u=-(n-2)e^{-2u}-2e^{-4u}$ in $\rr^n.$
Note that  $  -(n-2)e^{-2s}-2e^{-4s}\le 0$ in $\rr$.

Below,  we show that
$u=-\frac 12\log(1+|x|^2)$ is a stable solution to the equation
$$-\bdz u=f(u)\quad {\rm in}\quad \rr^n,$$
where $f(s)=(n-2)e^{2s}+2e^{4s}\ge 0$ in $\rr$.

First we show that $u$ is a solution. Indeed, for any $x\in\rn$ a direct calculation  gives
\begin{align*}
-\bdz u(x)&= \left((1+|x|^2)^{-1}x_i\right)_{x_i}= \frac{n}{1+|x|^2}+2\frac{|x|^2}{(1+|x|^2)^2}
= (n-2)\frac{1}{1+|x|^2}+2\frac{1}{(1+|x|^2)^2}.
\end{align*}
By 
 $e^{ 2u(x)}=(1+|x|^2)^{-1}$,
we have
\begin{align*}
-\bdz u(x)&=   (n-2)e^{2u(x)}+2e^{4u(x)} =f(u(x)).
\end{align*}

Next,    we  show that $u$ is stable. Note that $f'(s)=2(n-2)e^{2s}+8e^{4s} $ for $s\in\rr$.
 Given any $x\ne 0$, writing  $r=|x| $ and noting  $e^{ 2u(x)}=(1+|x|^2)^{-1}$, we  have
\begin{align*}
f'(u(x)) =2(n-2)e^{2u(x)}+8e^{4u(x)} =\frac{2(n-2)}{1+r^2}+\frac{8}{(1+r^2)^2}
\end{align*}
 By $n=10$ we have
\begin{align*}
f'(u(x)) &=\frac{16r^2(1+r^2)+8r^2}{ r^2(1+r^2)^2} 
 =\frac{16 r^4+24 r^2}{ r^2(1+r^2)^2} < \frac{16(1+r^2)^2}{{ r^2(1+r^2)^2} }=\frac{(n-2)^2}{4|x|^2}. 
\end{align*} 
 By this, and the Hardy inequality, we have 
$$\int_{\rr^n}f'(u)\xi^2\,dx\le \frac{(n-2)^2}{4}
\int_{\rr^n}\frac{\xi^2}{|x|^2}\,dx
\le \int_{\rr^n}|D\xi|^2\,dx\quad\forall \xi\in C^\fz_c(\rr^n).$$
 Thus   $u $ is a stable solution to
 $-\bdz u=f(u)$ in $\rn$.

\noindent  Fa Peng

\noindent
School of Mathematical Science, Beihang University, Changping District Shahe Higher Education Park
  South Third Street No. 9, Beijing 102206, P. R. China

\noindent{\it E-mail }:  \texttt{fapeng158@163.com}

\bigskip

\noindent Yi Ru-Ya Zhang

\noindent
Academy of Mathematics and Systems Science, the Chinese Academy of Sciences, Beijing 100190, P. R. China

\noindent{\it E-mail }:  \texttt{yzhang@amss.ac.cn}

\bigskip

\noindent  Yuan Zhou

\noindent
School of Mathematical Science, Beijing Normal University, Haidian District Xinjiekou Waidajie No.19, Beijing 10875, P. R. China

\noindent{\it E-mail }:  \texttt{yuan.zhou@bnu.edu.cn}

\end{document}